\documentclass{article}

\usepackage[utf8]{inputenc}
\usepackage{amsmath,amsthm,amssymb}
\usepackage{graphicx} 
\usepackage{xcolor}
\usepackage{lineno}
\usepackage{enumitem}
\usepackage[margin=3.5cm]{geometry}
\usepackage[normalem]{ulem}
\usepackage{thmtools,thm-restate}
\usepackage{comment}

\usepackage{tikz,verbatim}
\usetikzlibrary{shapes.geometric}
\newtheorem{theorem}{Theorem}[section]

\newtheorem{lemma}[theorem]{Lemma}
\newtheorem{conjecture}[theorem]{Conjecture}
\newtheorem{definition}[theorem]{Definition}
\newtheorem{prop}[theorem]{Proposition}

\newtheorem*{claim*}{Claim}
\newtheorem*{subclaim*}{Subclaim}
\newtheorem*{ack}{Acknowledgments}
\usepackage{pgfplots}

\usetikzlibrary{decorations.markings}
\tikzset{myptr/.style={decoration={markings,mark=at position 1 with %
    {\arrow[scale=3,>=stealth]{>}}},postaction={decorate}}}
\newcommand{\equal}{=} 
\tikzset{
blackvertexv2/.style={circle, draw=black!100,fill=black!100,thick, inner sep=0pt, minimum size= 2.5mm},
bluevertexv2/.style={circle, draw=black!100,fill=blue!100,thick, inner sep=0pt, minimum size= 2.5mm},
redvertexv2/.style={circle, draw=black!100,fill=red!100,thick, inner sep=0pt, minimum size= 2.5mm},
dummywhite/.style={circle, draw=white!100,fill=white!100,thick, inner sep=0pt, minimum size= 0.5mm},
ellipsenodev1/.style={ellipse, draw=black!100,fill=none,thick, inner sep=0pt, minimum width= 1.5cm, minimum height = 2.75cm},
ellipsenodev2/.style={ellipse, draw=black!100,fill=none,dashed, inner sep=0pt, minimum width= .75cm, minimum height = 2.25cm},
ellipsenodev3/.style={ellipse, draw=black!100,fill=none,thick, inner sep=0pt, minimum width= 2cm, minimum height = 1cm},
}
\definecolor{hanpurple}{rgb}{0.32, 0.09, 0.98}
\definecolor{SoffiaRed}{RGB}{160,20,20}

\newcommand{\lp}{\! \left (}
\newcommand{\rp}{\right )}

\newcommand{\vtxa}{u}
\newcommand{\vtxb}{v}
\newcommand{\vtxc}{w}

\newcommand{\oddbound}{{t \choose 2}(7t+7)}
\newcommand{\odd}{\xrightarrow{\text{odd}}}

\newenvironment{claimproof}{ \trivlist
	\item[\hskip\labelsep
	\textit{Proof of the claim}.]\ignorespaces
}{\hfill$\vartriangleleft$\medskip
	
}

\usepackage[colorlinks=true,linkcolor=red!80!black,citecolor=blue]{hyperref}
\usepackage{authblk}

\usepackage{graphicx} 

\usepackage{todonotes}

\title{Homomorphism counting from a proper immersion-closed class is not isomorphism}
\author[1,2]{Andrea Jim\'enez\thanks{A. Jim\'enez is
supported by ANID/Fondecyt Regular 1220071 and ANID-MILENIO-NCN2024-103.}}
\author[3]{Benjamin Moore\thanks{Benjamin Moore acknowledges the support of the Natural Sciences and Engineering Research Council of Canada (NSERC)[RGPIN-2025-07125],  Cette recherche a été financé par le Conseil de recherches en sciences naturelles et en génie du Canada (CRSNG) [RGPIN-2025-07125].}}

\author[1]{Daniel A. Quiroz\thanks{D.A. Quiroz is
supported by ANID/Fondecyt Regular 1252197 and MATH-AMSUD MATH230035. \phantom{arroz}Emails: \{Ben.Moore\}@umanitoba.ca, \{andrea.jimenez,daniel.quiroz\}@uv.cl,  yyoo2@alaska.edu}}
\author[4]{Youngho Yoo}

\affil[1]{Universidad de Valparaíso, Valparaíso, Chile}
\affil[2]{Millennium Nucleus for Social Data Science (SODAS), Santiago, Chile}
\affil[3]{University of Manitoba, Winnipeg, Canada}
\affil[4]{University of Alaska Fairbanks, Fairbanks, United States of America}
\date{February 2026}

\begin{document}

\maketitle

\begin{abstract}
    Lov\'{a}sz proved that two graphs $G$ and $H$ are isomorphic if $\hom(K,G) = \hom(K,H)$ for all graphs $K$, where $\hom(G_1,G_2)$ denotes the number of homomorphisms from $G_1$ to $G_2$. Dvo\v{r}\'{a}k showed that it suffices to count homomorphisms from all $2$-degenerate graphs $K$. On the other hand, for several interesting graph classes $\mathcal{M}$, it has been shown that there exist non-isomorphic graphs $G$ and $H$ such that $\hom(K,G)=\hom(K,H)$ for all $K\in \mathcal{M}$. Most such classes are minor-closed and Roberson conjectured that every proper minor-closed graph class $\mathcal{M}$ has the property that there exist non-isomorphic graphs that are indistinguishable by homomorphism counts from $\mathcal{M}$.  There has been an effort to prove Roberson's conjecture as it is believed that minor-closed classes play a special role in demarcating the graph classes that satisfy this property.

    We show that this special role, if it exists, must be shared, by proving an analogue of Roberson's conjecture for a rich family of non-minor-closed classes. Namely, we prove that for any proper immersion-closed graph class $\mathcal{M}$, there exist non-isomorphic graphs $G$ and $H$ such that $\hom(K,G) = \hom(K,H)$ for all $K \in \mathcal{M}$.
    This extends a result of Roberson on homomorphism indistinguishability over bounded degree graphs, and cannot be extended in the natural way by replacing immersions with topological minors due to a result of Neuen and Seppelt. 
    
    Our main result is obtained as a consequence of a colouring result which may be of independent interest: Every graph with a $\oddbound$-oddomorphism admits a $K_{t}$-immersion. We also give a shorter proof of a result of Neuen that every graph with a $t$-oddomorphism has treewidth at least $t-1$, which implies that homomorphism counts from graphs of bounded treewidth does not distinguish non-isomorphic graphs.

\end{abstract}
\section{Introduction}
Determining the complexity of graph isomorphism is one of the most important open problems in theoretical computer science. Currently, the best known worst-case running time of a graph isomorphism algorithm is quasi-polynomial due to Babai \cite{Bab16}. Babai's algorithm uses a combination of sophisticated algebraic techniques as well as the $k$-dimensional Weisfeiler-Leman algorithm for small values of $k$. Practical algorithms, such as NAUTY, all use a framework called individualization-refinement \cite{MCKAY201494}.
The solvers all build on a framework called individualization-refinement. While we will not need the precise details, they all revolve around finding fast ``refinement" algorithms, and then getting them ``unstuck" if needed with individualization arguments. 

The most well known of these refinement algorithm is called \textit{colour refinement} or the \textit{one-dimensional  Weisfeiler-Leman algorithm}. This refinement algorithm is used in practical graph isomorphism solvers and it was shown in \cite{Smoothedanalysisstoc} that, in a smoothed analysis setting, colour refinement solves the majority of cases of graph isomorphism in near-linear time. While the algorithmic definitions of colour refinement are more useful for actual computation, in this note we will instead consider the homomorphism definition. 

Recall that for two graphs $G$ and $H$, a \underline{graph homomorphism} from $G$ to $H$ is a function $f:V(G) \to V(H)$ such that for all $e = xy \in E(G)$, we have $f(x)f(y) \in E(H)$. Let $\hom(G,H)$ denote the number of homomorphisms from $G$ to $H$. 

For two graphs $G$ and $H$, we say that the colour refinement algorithm \underline{fails to distinguish $G$} \underline{from $H$} if for all trees $T$, we have $\hom(T,G) = \hom(T,H)$. For those used to the algorithmic version of colour refinement, this is perhaps a surprising equivalence, which was shown by Dvo\v{r}\'{a}k in \cite{Dvorak}. Nevertheless, we  take this as a definition of the colour refinement.

While colour refinement is extremely effective in practice, it also fails on relatively simple classes of graphs. For example, it is easy to see it cannot distinguish two $d$-regular graphs on the same number of vertices. In particular, it fails to distinguish $2K_{3}$ from $C_{6}$, i.e., for every tree $T$, we have $\hom(T,2K_{3}) = \hom(T,C_{6})$. Therefore, a natural approach to extending the colour refinement algorithm is to count homomorphisms from wider classes of graphs beyond trees. This produces spectacular results and often maintains the efficiency of colour refinement.

Given a class $\mathcal{M}$ of graphs, we say that two graphs $G$ and $H$ are \underline{$\mathcal{M}$-isomorphic}, denoted $G \equal_{\mathcal{M}} H$, if $\hom(K,G) = \hom(K,H)$ for all graphs $K \in \mathcal{M}$. Lov\'{a}sz famously proved:

\begin{theorem}[\cite{Lovasz1967Operations}]
\label{thm:lovasz}
    If $\mathcal{M}$ is the class of all graphs, then $G \equal_{\mathcal{M}} H$  if and only if $G$ and $H$ are isomorphic. 
\end{theorem}

Thus, homomorphism counts from trees produce the colour refinement algorithm and homomorphism counts from all graphs produce graph isomorphism. Homomorphism counts from other classes in between produce interesting results as well. Table \ref{tab:homomorphism_equivalence} shows some of the various relationships, but is by no means exhaustive. 

\begin{table}[h]
\centering
\begin{tabular}{|l|p{5cm}|p{4cm}|}
\hline
\textbf{Graph class $\mathcal{M}$} & \textbf{$G \equal_{\mathcal{M}} H$} & \textbf{Complexity}  \\
\hline
Cycles & $G$ and $H$ are cospectral (see \cite{GroheRattanSeppelt2025}) & $O(n^{3})$ (see \cite{GroheRattanSeppelt2025})  \\
\hline
Graphs with treewidth $\leq k$ & $G$ and $H$ are indistinguishable by the $k$-dimensional Weisfeiler-Leman Algorithm (see \cite{GroheRattanSeppelt2025}) & $O(n^{k+1}\log(n))$  \cite{immerman2019kdimensionalweisfeilerlemanalgorithm}) \\
\hline
Planar graphs & $G$ and $H$ are quantum isomorphic \cite{quantumisomorphism} & Undecidable \cite{ATSERIAS2019289} \\
\hline
2-degenerate graphs & $G$ and $H$ are isomorphic \cite{Dvorak} & At worst quasi-polynomial time \cite{ Bab16} \\
\hline
All graphs & $G$ and $H$ are isomorphic \cite{Lovasz1967Operations} & At worst quasi-polynomial time \cite{Bab16} \\
\hline
Treedepth $\leq k$ & $G$ and $H$ admit a coKleisli isomorphism with respect to Ehrenfeucht-Fraïss\'{e} comonad \cite{DawarJaklReggio2021} & LOGSPACE \cite{treedepth} \\
\hline
Disjoint union of paths &  $G$ and $H$'s adjacency matrices have the same multisets of main eigenvalues \cite{cerny2025homomorphismindistinguishabilitymultiplicityautomata} & Exact Logspace \cite{cerny2025homomorphismindistinguishabilitymultiplicityautomata} \\ 
\hline
$K_{5}$-topological-minor-free & $G$ and $H$ are isomorphic \cite{neuen2026distinguishinggraphscountinghomomorphisms}& At worst quasi-polynomial time \cite{Bab16} \\
\hline
\end{tabular}
\caption{Homomorphism counting equivalence for different graph classes. For each class $\mathcal{M}$, the table shows the characterization of $G \equal_{\mathcal{M}} H$ and computational complexity of the equivalence test.}
\label{tab:homomorphism_equivalence}
\end{table}

We see from Table \ref{tab:homomorphism_equivalence} that for a given graph class $\mathcal{M}$, it is not obvious what the relation $\equal_{\mathcal{M}}$ is, nor what the complexity of determining if $G \equal_{\mathcal{M}} H$ is. However, important graph isomorphism tests appear by considering various natural classes $\mathcal{M}$.  In Table~\ref{tab:homomorphism_equivalence}, every class~$\mathcal{M}$ for which $=_\mathcal{M}$ differs from isomorphism is contained in a minor-closed\footnote{Given a graph $G$, a graph $H$ is a \underline{minor} of $G$ if it can be obtained by deleting edges, contracting edges, and deleting isolated vertices. A graph class $\mathcal{M}$ is \underline{minor-closed} if every minor of every graph in $\mathcal{M}$ is also in $\mathcal{M}$.} class. This motivates the following conjecture of Roberson.
A graph class $\mathcal{M}$ is \underline{proper} if it does not contain all graphs and it is \underline{union-closed} if the disjoint union of any two graphs in $\mathcal{M}$ is also in $\mathcal{M}$.

\begin{conjecture}[\cite{oddomorphismpaper}]\label{conj:roberson}
    If $\mathcal{M}$ is a proper graph class that is minor-closed and union-closed, then there are non-isomorphic graphs $G$ and $H$ such that $G\equal_{\mathcal{M}} H$.
\end{conjecture}

However, even if this conjecture were true, it would not give the full picture. Indeed if $\mathcal{M}$ is the class of graphs with maximum degree at most $d$, then $\equal_{\mathcal{M}}$ is not the isomorphism relation as shown in \cite{oddomorphismpaper}, and the class of graphs with maximum degree at most $d$ is not minor-closed for any fixed $d \geq 3$. So what is the correct level of generality? A natural starting point to unify bounded degree graphs and minor-closed classes are classes that are closed under topological minors\footnote{$G$ contains $H$ as a \underline{topological minor} if a subgraph of $G$ contains a subdivision of $H$}. Yet recently Neuen and Seppelt~\cite{neuen2026distinguishinggraphscountinghomomorphisms} showed that homomorphism counting is isomorphism even for the class of graphs with no $K_{5}$ topological minor. The authors of \cite{neuen2026distinguishinggraphscountinghomomorphisms} interpret this last result as evidence of the ``special role'' that  minor-closed classes play in demarcating the graph classes $\mathcal{M}$ for which $=_\mathcal{M}$ differs from isomorphism. However, we show that the role of \textit{immersions} is at least as important as minors by proving the following immersion-analogue of Roberson's conjecture. 

\begin{theorem}
\label{thm:maintheorem}
    If $\mathcal{M}$ is a proper graph class that is immersion-closed and union-closed, then there are non-isomorphic graphs $G$ and $H$ such that for all $K \in \mathcal{M}$, we have $\hom(K,G) = \hom(K,H)$. 
\end{theorem}

While a proper immersion-closed graph class $\mathcal{M}$ may not be union-closed, it is contained in a proper graph class $\mathcal{M}'$ that is immersion-closed and union closed. In particular, if $G$ and $H$ are non-isomorphic graphs such that $\hom(K,G)=\hom(K,H)$ for all $K\in \mathcal{M}'$, then we also have $\hom(K,G)=\hom(K,H)$ for all $K\in \mathcal{M}$. Hence, Theorem \ref{thm:maintheorem} holds for every proper immersion-closed graph class.

Immersion is a well-studied\footnote{See e.g. \cite{BUSTAMANTE2022103550,CHUDNOVSKY2016208,kuratowskiimmersion,LIU2023252,WOLLAN201547}, as a small sampling.}  graph order incomparable to the minor order, which nevertheless shares important properties of it: it is a well-quasi-order, as conjectured by Nash-Williams in 1963 \cite{NashWilliams1963} and proved by Robertson and Seymour~\cite{ROBERTSON2010181}. Moreover, if a graph $G$ contains $H$ as a topological minor, then $G$ contains $H$ both as a minor and as an immersion. In particular, this tells us that the result of Neuen and Seppelt helps to emphasize the importance of not only Conjecture \ref{conj:roberson} but Theorem~\ref{thm:maintheorem} as well, since graph classes closed under topological minors are natural generalizations of both immersion-closed classes and minor-closed classes. We now give the definition of immersion.

\begin{definition}
    Let $G$ be a graph, and suppose we have three distinct vertices $u,v,w$ with edges $uv$,$vw$. To \underline{split off} $uv$ and $vw$ is to delete the edges $uv$ and $vw$, and add the edge $uw$ (possibly adding a parallel edge if necessary).
     A graph $G$ contains a graph $H$ as an \underline{immersion} if $H$ can be obtained from a subgraph of $G$ by splitting off pairs of edges and deleting isolated vertices.
     A graph class $\mathcal{M}$ is \underline{immersion-closed} if for every $G\in \mathcal{M}$ and every graph $H$, if $G$ contains $H$ as an immersion, then $H\in \mathcal{M}$.
\end{definition}

Note that for every positive integer $d$, the class of graphs with maximum degree at most $d$ is immersion-closed. Therefore, Theorem~\ref{thm:maintheorem} generalizes the result of Roberson \cite{oddomorphismpaper} that homomorphisms from bounded degree classes does not capture isomorphism.  The following further strengthening may be possible and we leave it as an open problem. 

\begin{conjecture}
    Let $\mathcal{M}_{1}$ and $\mathcal{M}_{2}$ be two distinct immersion-closed and union-closed classes. Then for some $i,j \in \{1,2\}$ with $i \neq j$, there are graphs $G$ and $H$ such that $G \equal_{\mathcal{M}_{i}} H$ and $G  \neq_{\mathcal{M}_{j}} H$.
\end{conjecture}

To prove Theorem \ref{thm:maintheorem}, we utilize tools developed by Roberson in \cite{oddomorphismpaper}, which essentially reduce the problem to a colouring one. 
Recall that $N_{G}(v)$ denotes the neighbourhood of $v$ in a graph $G$. The following definitions are introduced in  \cite{oddomorphismpaper}.

\begin{definition}
    Let $f\colon V(G)\rightarrow S$ be a proper colouring of a graph $G$. A vertex $\vtxa \in V(G)$ is \underline{$f$-odd} (respectively, \underline{$f$-even}) if $N_G(v)$ contains an odd (respectively, even) number of vertices of each colour in $S\setminus \{ f(v)\}$. 
\end{definition}

Of course, it is possible for a vertex to be neither $f$-odd nor $f$-even. 

\begin{definition}
A proper $t$-colouring $f$ of a graph $G$ is a \underline{$t$-oddomorphism}, if
\begin{enumerate}
    \item Every vertex of $G$ is either $f$-even or $f$-odd.
    \item Every colour class contains an odd number of odd vertices.
\end{enumerate}
\end{definition}

If $G$ admits a $t$-oddomorphism $f$, we say that $f$ is an \underline{oddomorphism from $G$ to $K_t$} and write $G \odd K_{t}$. In Section~\ref{sec:mainresult} we prove the following.

\begin{theorem}
\label{thm:largeoddimplieslargeimmersion}
    If $G\odd K_{\oddbound}$, then $G$ contains a $K_{t}$-immersion. 
\end{theorem}

While it does not appear at all related, it turns out that Theorem \ref{thm:largeoddimplieslargeimmersion}   and a result in \cite{oddomorphismpaper} implies Theorem \ref{thm:maintheorem}. We show that this is the case in Section~\ref{sec:homcounting}. The bound in Theorem~\ref{thm:largeoddimplieslargeimmersion} is not necessarily optimal, and we make the following conjecture. 

\begin{conjecture}
    There exists an absolute constant $c$ such that if $G\odd K_{ct}$, then $G$ contains a $K_{t}$-immersion. 
\end{conjecture}

It might even be possible that $c=1$ can be taken.
When $t\leq 4$, we can indeed take $c=1$ due to a tight lower bound on the treewidth $\operatorname{tw}(G)$ of graphs $G$ admitting a $t$-oddomorphism below. (We recall the definition of treewidth in Section~\ref{sec:prelims}.)
While Theorem \ref{thm:treewidththm} follows from the proof of \cite[Theorem 1.2]{Neuen2024HomomorphismDistinguishing} by Neuen,\footnote{Neuen proves a more general result that if $G$ admits an oddomorphism to \textit{any} graph $H$ (see Definition \ref{def:oddomorphismGH}), then $\operatorname{tw}(G)\geq \operatorname{tw}(H)$.} we give a short elementary proof of this fact. 

\begin{theorem}
\label{thm:treewidththm}
     If $f:G \odd K_{t}$, then $\operatorname{tw}(G) \geq t-1$. 
\end{theorem}

The $t-1$ bound given here is best possible since $K_{t}\odd K_{t}$ and $\operatorname{tw}(K_{t})=t-1$. Theorem~\ref{thm:treewidththm} also gives an alternative proof that homomorphism counts from graphs of bounded treewidth graph do not fully capture isomorphism. However, as we have mentioned above, homomorphism counts from bounded treewidth graphs is exactly described by the $k$-dimensional Weisfeiler-Leman algorithm, which is already known to not fully capture isomorphism (\cite{CaiFurerImmerman1992}). 
The rest of the paper is structured as follows. In Section \ref{sec:prelims} we give some definitions, state a known key result, and prove some preliminary lemmas about oddomorphisms. In Section \ref{sec:mainresult}, we prove Theorem \ref{thm:largeoddimplieslargeimmersion}. In Section \ref{sec:treewidth}, we prove Theorem \ref{thm:treewidththm}. Finally, in Section \ref{sec:homcounting}, we prove Theorem \ref{thm:maintheorem}. We devote the rest of this section to sketch the proofs of our oddomorphism results.

To prove Theorem \ref{thm:largeoddimplieslargeimmersion}, we observe first that oddomorphisms are preserved when deleting a $2$-coloured cycle and when splitting off a $2$-coloured path between two $f$-odd vertices of different colours. Thus, in a minimal counterexample, any two colour classes induce a forest in which each component has $f$-odd vertices of exactly one colour. We further split off 2-coloured paths between $f$-odd vertices in these components to obtain a new graph $G'$. Note that the oddomorphism is no longer preserved. Nonetheless, if $f$ is a $\oddbound$-oddomorphism of $G$, then either $G'$ contains two vertices of same colour with at least ${t \choose 2}$ parallel edges between them, or the graph obtained from $G'$ by removing parallel edges has minimum degree at least $7t+7$. 

In the former case, we identify the two vertices in $G$, delete the $2$-coloured paths between them, and take the symmetric difference of their neighbourhoods. 
This gives a smaller graph with an $\oddbound$-oddomorphism, and we show that the $K_{t}$-immersion of the smaller graph lifts to an immersion of the original graph. 
In the latter case, a theorem of Gauthier, Le and Wollan \cite{GAUTHIER201998} gives the $K_{t}$-immersion.


To prove Theorem \ref{thm:treewidththm}, suppose $G$ has a $t$-oddomorphism $f$ and has treewidth at most $t-2$, and that $G$ is vertex minimal with respect to these properties. Take a tree decomposition of minimum width and root the tree decomposition at some vertex. We choose a non-leaf bag $\beta_v$ containing an $f$-odd vertex $w$ such that no other $f$-odd vertex appears farther from the root than $w$. Since $f$ is a $t$-oddomorphism and each bag has size at most $t-1$, $w$ has a neighbour $w'$ in a bag farther away from the root whose colour does not appear in $\beta_v$. By looking at the Kempe chain between these two colours, we find two vertices with the same colour in $\beta_v$ that we can identify without increasing the treewidth while preserving the oddomorphism, contradicting the minimality of $G$.

\section{Preliminaries}
\label{sec:prelims}

\subsection{An equivalent definition for immersions and a key result}

An equivalent (see \cite{immersiondef}) and particularly  useful definition of immersions reads as follows. A graph $G$ contains a graph $H$ as an \underline{immersion} if there is a one-to-one mapping $f :V(H) \to V(G)$ and a collection of edge-disjoint paths in $G$, one for each edge
of $H$, such that the $(u,v)$-path $P_{uv}$ corresponding to the edge $uv \in E(H)$ has endpoints $f(u)$
and $f(v)$. The vertices in $f(V(H))$ are called \underline{branch vertices}. 

We will need the following theorem of Gauthier, Le and Wollan. (See \cite{immersiondef} and \cite{DeVosDvorakFoxMcDonaldMoharScheide2014} for previous results on minimum degree conditions forcing clique immersions.)

\begin{theorem}[\cite{GAUTHIER201998}]
\label{thm:mindegreethm}
    Every simple graph with minimum degree $7t +7$ has a $K_{t}$-immersion. 
\end{theorem}

\subsection{Definition of treewidth}

Let $G$ be a graph. A \underline{tree decomposition} of $G$ is a pair
$\left(T, \{\beta_x\}_{x \in V(T)}\right)$ where $T$ is a tree and each $\beta_x \subseteq V$
(called a \underline{bag}) such that:
\begin{enumerate}
  \item $\bigcup_{x \in V(T)} \beta_x = V$;
  \item for every edge $uv \in E(G)$, there exists $x \in V(T)$ with
        $uv \in E(G[\beta_x])$;
  \item for every vertex $v \in V(G)$, the set
        $\{ x \in V(T) \mid v \in \beta_x \}$ induces a connected subtree of $T$.
\end{enumerate}

The \underline{width} of a tree decomposition
$\left(T, \{\beta_x\}_{x \in V(T)}\right)$ is $\max_{x \in V(T)} |\beta_x| - 1$. The \underline{treewidth} of a graph $G$, denoted $\operatorname{tw}(G)$, is the minimum
width over all tree decompositions of $G$.

\subsection{Preliminaries about oddomorphisms}
First we record a property that is immediate from the definition, and that we will use implicitly throughout.

\begin{prop}\label{prop:mindegree}
Let $f:G \odd K_{n}$. Then every $f$-odd vertex has degree at least $n-1$.
\end{prop}

\begin{proof}
    An $f$-odd vertex has at least one neighbour in each colour class other than its own colour, and hence has degree at least $n-1$. 
\end{proof}

Now we give an important operation which preserves oddomorphisms.

\begin{definition}

Given an oddomorphism $f: G \odd K_n$, two vertices $\vtxa_1, \vtxa_2$ such that $f(\vtxa_{1}) = f(\vtxa_2)$ and a (possibly empty) set $\mathcal{P}$ of pairwise edge-disjoint $(\vtxa_{1},\vtxa_{2})$-paths, such for each path $P \in \mathcal{P}$, $P$ is $2$-coloured under $f$, we define the \underline{$\mathcal{P}$-merger of $\vtxa_1$ and $\vtxa_2$} to be the graph $G'$ obtained by identifying $\vtxa_1$ and $\vtxa_2$ into a single vertex $\vtxa'$ whose neighbourhood is the symmetric difference of the neighbourhoods of $\vtxa_1$ and $\vtxa_2$, and then deleting all of the edges from the paths in $\mathcal{P}$. 
\end{definition}
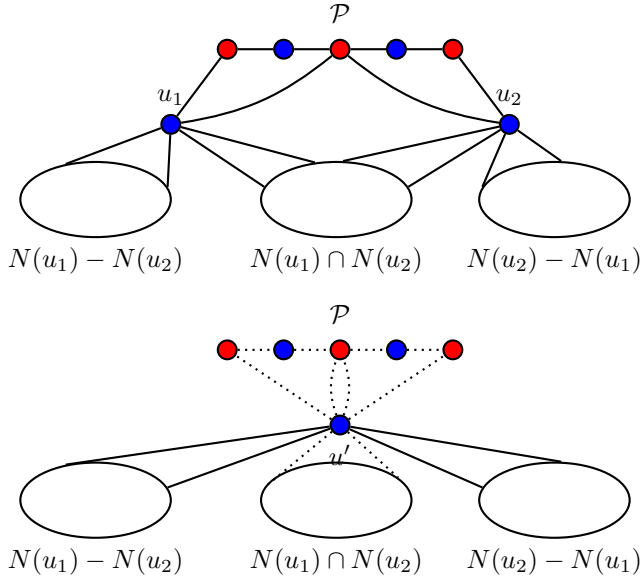
\begin{figure}
\begin{center}
\begin{tikzpicture}
    \node[bluevertexv2] at (0,0) (v1) [label = above:$\vtxa_{1}$] {};
    \node[bluevertexv2] at (4.5,0) (v2) [label = above:$\vtxa_{2}$] {};
    \node[redvertexv2] at (0.75,1) (v3) {};
    \node[bluevertexv2] at (1.5,1) (v4) {};
    \node[redvertexv2] at (2.25,1) (v5) {};
    \node[bluevertexv2] at (3,1) (v6) {};
    \node[redvertexv2] at (3.75,1) (v7) {};
    \node[dummywhite] at (2.25,1.2) (dummy1) [label = above:$\mathcal{P}$] {};
    \draw[thick,black] (v1)--(v3)--(v4)--(v5)--(v6)--(v7)--(v2);
    \draw[thick,bend right=15] (v1) to (v5);
    \draw[thick,bend right=15] (v5) to (v2);
    \node[ellipsenodev3] at (-1,-1) (ellipse1) [label = below:$N(u_{1}) - N(u_{2})$]{};
    \draw[thick,black] (v1) to (ellipse1.130);
    \draw[thick,black] (v1) to (ellipse1.10);
    \node[ellipsenodev3] at (2.2,-1) (ellipse2) [label = below:$N(u_{1}) \cap N(u_{2})$] {};
    \draw[thick,black] (v1) to (ellipse2.170);
    \draw[thick,black] (v1) to (ellipse2.120);
    \draw[thick,black] (v2) to (ellipse2.80);
    \draw[thick,black] (v2) to (ellipse2.10);
    \node[ellipsenodev3] at (5.1,-1) (ellipse3) [label = below:$N(u_{2})- N(u_{1})$] {};
    \draw[thick,black] (v2) to (ellipse3.170);
    \draw[thick,black] (v2) to (ellipse3.80);
    \begin{scope}[yshift = -4cm]
    \node[redvertexv2] at (0.75,1) (v3) {};
    \node[bluevertexv2] at (1.5,1) (v4) {};
    \node[redvertexv2] at (2.25,1) (v5) {};
    \node[bluevertexv2] at (3,1) (v6) {};
    \node[redvertexv2] at (3.75,1) (v7) {};
    \node[dummywhite] at (2.25,1.2) (dummy1) [label = above:$\mathcal{P}$] {};
    \node[bluevertexv2] at (2.25,0) (u1) [label= below:$u'$] {};
    \draw[thick,black,dotted] (u1) to (v3);
    \draw[thick,black,dotted] (u1) to (v7);
    \draw[thick,black,dotted] (v3)--(v4)--(v5)--(v6)--(v7);

    \node[ellipsenodev3] at (-1,-1) (ellipse1) [label = below:$N(u_{1}) - N(u_{2})$]{};
    \draw[thick,black] (u1) to (ellipse1.130);
    \draw[thick,black] (u1) to (ellipse1.10);
    \node[ellipsenodev3] at (2.2,-1) (ellipse2) [label = below:$N(u_{1}) \cap N(u_{2})$] {};
    \draw[thick,black,dotted] (u1) to (ellipse2.170);
    \draw[thick,black,dotted] (u1) to (ellipse2.20);

    \node[ellipsenodev3] at (5.1,-1) (ellipse3) [label = below:$N(u_{2})- N(u_{1})$] {};
    \draw[thick,black] (u1) to (ellipse3.170);
    \draw[thick,black] (u1) to (ellipse3.80);

    \draw[thick,black, bend right = 20,dotted] (u1) to (v5);
    \draw[thick,black, bend left =20,dotted] (u1) to (v5);
    \end{scope}
\end{tikzpicture}
\caption{An example of a $\mathcal{P}$-merger of $u_{1}$ and $u_{2}$. Note that the paths in $\mathcal{P}$ may also be in the union of $N(u_{1})$ and $N(u_{2})$.}
\label{fig:Merger}
\end{center}
\end{figure}
In the event $\mathcal{P}$ is empty, we simply call this the \underline{merger} of $u_{1}$ and $u_{2}$.
See Figure \ref{fig:Merger} for an illustration.
\begin{lemma}
\label{lem:mergerlemma}
    Suppose that $f: G \odd K_n$ and $G'$ is obtained from $G$ by the $\mathcal{P}$-merger of vertices $\vtxa_1$ and $\vtxa_2$ into $\vtxa'.$ Let $f':V(G) \to V(K_n)$ be the map where $f'(\vtxa') = f(\vtxa_1) =f(\vtxa_2)$ and, for all $\vtxb \in V(G) - \{\vtxa_1,\vtxa_2\}$, we have $f'(\vtxb) = f(\vtxb)$. Then we have $f': G \odd K_n$. 
\end{lemma}

\begin{proof}
Assume without loss of generality that $V(K_n)=\{1,\dots,n\}$ and for $c\in V(K_n)$, let $C_c=f^{-1}(\{c\})$ and $C_c'=(f')^{-1}(\{c\})$.
Observe that $f'$ is a homomorphism as $\vtxa_1$ and $\vtxa_2$ have the same colour under $f$. Let $c = f \lp \vtxa' \rp$.

First consider a vertex $\vtxb \in V \lp G' \rp - \vtxa'$ and let $P_1,\dots,P_\ell$ denote the paths in $\mathcal{P}$ containing $v$ (where possibly $\ell=0$). Let $c'=f(v)$ (where possibly $c'=c$). 
If $c' \neq c$, then as each path in $\mathcal{P}$ is $2$-coloured and has endpoints $\vtxa_{1}$ and $\vtxa_{2}$, every neighbour of $\vtxb$ in $P_{1},\ldots,P_{\ell}$ is coloured $c$. Thus $|N_{G'}(v)\cap C_c'|=|N_G(v)\cap C_c|- (2\ell -2b)$, where $b \in \{0,1\}$ and $b=1$ if and only if $\vtxb$ is adjacent to both $\vtxa_{1}$ and $\vtxa_{2}$ and the path $\vtxa_{1},\vtxb,\vtxa_{2}$ is not in $\mathcal{P}$. 
Therefore $|N_{G'}(v)\cap C_c'|$ has the same parity as $|N_G(v)\cap C_c|$.

Now assume that $c' =c$. Then $\vtxb$ is adjacent to neither $\vtxa_{1}$ nor $\vtxa_{2}$. For each colour $c''$, let $q_{c''}$ be the number of paths in $\mathcal{P}$ that contain $v$ and a vertex coloured $c''$. Then $|N_{G'}(v)\cap C_{c''}'| = |N_G(v)\cap C_{c''}|-2q_{c''}$, hence $|N_{G'}(v)\cap C_c'|$ has the same parity as $|N_G(v)\cap C_c|$.

Now consider $\vtxa'$. Fix a colour $c' \neq c$. Let $q_{c'}$ be the number of paths in $\mathcal{P}$ which contain a vertex coloured $c'$, and let $p_{c'}$ be the number of vertices in the common neighbourhood of $\vtxa_{1}$ and $\vtxa_{2}$ coloured $c'$ which do not appear in a path in $\mathcal{P}$. Then 
\[|N_{G'}(u')\cap C_{c'}'| = |N_G(u_1)\cap C_{c'}|+|N_G(u_2)\cap C_{c'}|  -2q_{c'} - 2p_{c'}\]

Therefore if $\vtxa_{1}$ and $\vtxa_{2}$ are both $f$-even or both $f$-odd, then $\vtxa'$ is $f'$-even, and if exactly one of $\vtxa_{1}$ and $\vtxa_{2}$ is $f$-even, then $\vtxa'$ is $f'$-odd.

Therefore all vertices in $G'$ are either $f'$-even or $f'$-odd. The last remaining condition is that there are an odd number of $f'$-odd vertices of each colour class. Other than colour $c$, the number of $f$-odd vertices remains the same in $f$ and $f'$. For colour $c$, if both $\vtxa_{1}$ and $\vtxa_{2}$ were $f$-odd, then $\vtxa'$ is $f$-even, and there was at least three and an odd number of $f$-odd vertices of colour $c$ in $f$. Therefore there is an odd number of $f$-odd vertices of colour $c$ under $f'$, and hence $f'$ is an oddomorphism.  
\end{proof}

We also observe that removing bi-coloured cycles preserves oddomorphisms. For the next lemma, we use the notation that given two disjoint sets of vertices, $A,B \subseteq V(G)$ the graph $G[A,B]$ is the graph containing all edges with one endpoint in $A$ and one endpoint in $B$. 

\begin{lemma}
\label{lem:acycliccolouring}
Let $f: G \odd K_{n}$ and let $C_{i}$ and $C_{j}$ be two colour classes, $i \neq j$. Let $C$ be a cycle in $G[C_{i},C_{j}]$. Then $f: G-E(C) \odd K_{n}$.
\end{lemma}

\begin{proof}
    As we only deleted edges, $f$ is a homomorphism from $G-E(C)$ to $K_{n}$. The parity of every vertex remains the same in $G -E(C)$ as $G$, since $C$ is $2$-coloured, so the number of neighbours coloured $i$ or $j$ changes either by $0$ or $2$. It follows that $f: G-E(C) \odd K_{n}$. 
\end{proof}

In particular, Lemma \ref{lem:acycliccolouring} implies that if an oddomorphism is edge-minimal, it is a so-called acyclic colouring.

\section{Large oddomorphism implies large clique immersion}
\label{sec:mainresult}

In this section we prove Theorem \ref{thm:largeoddimplieslargeimmersion}.
Given an $(x,y)$-path $P$, we say that \underline{to split a path $P$} is to iteratively split off edges of $P$ to create the edge $xy$.

    Suppose, for a contradiction, that Theorem~\ref{thm:largeoddimplieslargeimmersion} is false, and consider a vertex-minimal, and subject to that, edge-minimal counterexample. By Lemma \ref{lem:acycliccolouring}, we may assume that $G[C_{i},C_{j}]$ is a forest for all $i \neq j$. 

    \begin{claim*}

        Let $C_{i}$ and $C_{j}$ be two distinct colour classes of $f$. Then there is no path of length at least $2$ from an odd vertex in $C_{i}$ to an odd vertex in $C_{j}$ in $G[C_{i},C_{j}]$.
    \end{claim*}

    \begin{claimproof}
        Suppose there is such an $(x,y)$-path $P$ where $x$ is an odd vertex in $C_{i}$ and $y$ is an odd vertex in $C_{j}$. Let $G'$ be the graph obtained from $G$ by splitting off the path $P$. We claim that $f$ is also an oddomorphism of $G'$. Indeed, each internal vertex of $P$ loses two neighbours of the same colour, and hence has the same parity. The vertex $x$ (respectively, $y$) loses one neighbour of colour $j$ (respectively, $i$) and gains a neighbour of colour $j$ (respectively, $i$). Thus the parity of $x$ and $y$ remain the same. All vertices not in $P$ have the same parity, and hence $f$ is an oddomorphism. As $P$ has length at least $2$, $G'$ has fewer edges than $G$. By the minimality of $G$, $G'$ has a $K_{t}$-immersion. As $G$ has $G'$ as an immersion, it follows that $G$ has a $K_{t}$-immersion. 
    \end{claimproof}

   Thus we have that for every two colour classes $C_{i}$ and $C_{j}$, the graph $G[C_{i},C_{j}]$ is a forest. Note that a vertex in $C_i\cup C_j$ has odd degree in $G[C_i,C_j]$ if and only if it is $f$-odd. It follows that we can decompose $G[C_{i},C_{j}]$ into (edge-disjoint) paths such that each path has two $f$-odd vertices in $G[C_{i},C_{j}]$ as its endpoints.

   Let $G'$ be the graph obtained by splitting off each such path for all pairs $i,j$ with $i \neq j$ and removing all even degree vertices. Let $G''$ be the simplification of $G'$ -- that is, $G''$ is the graph obtained from $G'$ by removing all parallel edges. If $G''$ has minimum degree at least $7t+7$, then by Theorem \ref{thm:mindegreethm}, $G''$ has a $K_{t}$-immersion, and as $G$ has $G''$ as an immersion, it follows that $G$ has a $K_{t}$-immersion. Thus we can assume that $G''$ has minimum degree less than $7t+7$. Now by Proposition \ref{prop:mindegree}, each odd vertex has at least one neighbour in each colour other than itself. Thus each odd vertex in $G'$ has degree at least ${t \choose 2}(7t+7)-1$ in $G'$. As $G''$ has minimum degree less than $7t+7$, this implies that there are two vertices $x,y$ such that there are at least ${t \choose 2}$ parallel edges between them, as otherwise vertices could have degree at most $({t \choose 2}-1)(7t+7)$.

   Let $x,y$ be two such vertices. By construction, there are at least ${t \choose 2}$ edge-disjoint paths $\mathcal{P} = \{P_{1},P_{2},\ldots,P_{{t \choose 2}}\}$ from $x$ to $y$ such that each path is $2$-coloured under $f$.

   By the claim if two odd vertices have different colours, either they are adjacent or there is no path between them using just two colours. By this observation and our construction it follows that $f(x) =f(y)$. 

   Let $G^{*}$ denote the $\mathcal{P}$-merger of $x$ and $y$ and let $f^*$ be the $\binom{t}2(7t+7)$-oddomorphism of $G^*$ as in  Lemma \ref{lem:mergerlemma}. 

   By the minimality of $G$, $G^{*}$ has a $K_{t}$-immersion. We claim that this implies that $G$ has a $K_{t}$-immersion. To see this, we use the definition of immersions given in Section~\ref{sec:prelims}. Let $t_{1},\ldots,t_{t}$ be the branch vertices of the $K_{t}$-immersion in $G^{*}$ and, for each $1\le i<j\le t$, let $Q_{ij}$ be the path from $t_{i}$ to $t_{j}$ in $G^*$ (such that $Q_{ij}$'s are mutually edge-disjoint). Observe that if $Q_{ij}$ uses the merged vertex more than once, $Q_{ij}$ in fact contains a cycle. Thus each path $Q_{ij}$ uses the merged vertex at most once. Therefore we can extend this immersion to an immersion of $G$ in the following manner. For each $Q_{ij}$ that uses the merged vertex, we use one of the paths $P_{1},\ldots,P_{{t \choose 2}}$ to extend $Q_{i,j}$ to a path in $G$. As there are at most ${t \choose 2}$ paths $Q_{ij}$, we can do this such that each $P_{i}$ is used at most once, and thus we have a $K_{t}$-immersion in $G$.

\section{A tight bound for oddomorphisms and treewidth}
\label{sec:treewidth}
In this section we prove Theorem \ref{thm:treewidththm}.

Suppose towards a contradiction that $\operatorname{tw}(G) \leq t-2$, $G$ has a $t$-oddomorphism $f$, and subject to these conditions, $G$ minimizes $|V(G)|+|E(G)|$. As removing edges does not increase the treewidth, Lemma \ref{lem:acycliccolouring} implies that for any two colour classes $C_{i},C_{j}$, we have that $G[C_{i},C_{j}]$ is acyclic.  Let $(T,\{\beta_x\}_{x\in V(T)})$ be a tree decomposition of $G$ of width at most $t-2$. Fix a vertex $r \in V(T)$ and view $T$ as a tree rooted at $r$. 

    For every $f$-odd vertex $u$ in $G$, let $D(u,r)$ be the shortest distance in $T$ from $r$ to a bag containing $u$. 
    Let $w$ be an $f$-odd vertex which maximizes $D(u,r)$ over all $f$-odd vertices $u$ (picking arbitrarily if there are more than one).
    Let $\beta_{v}$ be the bag which attains $D(w,r)$. 
    Note that $v\neq r$ because there are at least $t$ $f$-odd vertices and $|\beta_r|\leq t-1$.

    As $w$ is $f$-odd, for each of the $t-1$ colours distinct from $f(w)$, $w$ has a neighbour of that colour. Since $w\in\beta_v$ and $|\beta_v|\leq t-1$, there exists a colour $c'$ such that no vertex in $\beta_{v}$ has colour $c'$. Let $w'$ be a neighbour of $w$ of colour $c'$.
    Note that if $w'$ appears in a bag of the component of $T-v$ containing $r$, then $w$ also appears in that bag by the definition of tree decompositions, contradicting our choice of $\beta_v$. Hence, $w'$ appears in a bag of a component of $T-v$ not containing $r$. Let $c=f(w)$.

    Observe that in $G[C_{c},C_{c'}]$ there is a path starting at $w$, going through $w'$ and ending in some $f$-odd vertex $w''$. Let $P$ be such a path. If $P-w$ intersects $\beta_{v}$, then by our choice of the colour $c'$, there is a vertex $q \in \beta_{v}$ such that $q \neq w$ and $q$ has colour $c$. In this case, consider the merger of $q$ and $w$ and let $G'$ be the resulting graph. Then $|V(G')|<|V(G)|$ and, by Lemma \ref{lem:mergerlemma}, $G'$ admits a $t$-oddomorphism. 
    For each $x\in V(T)$, let $\beta_x'$ be the bag obtained from $\beta_x$ by replacing $q$ and $w$ (if they appear) with the merged vertex.
    Since $q$ and $w$ are in a common bag of $(T,\{\beta_x\}_{x\in V(T)})$, we have that $(T,\{\beta_x'\}_{x\in V(T)})$ is a tree decomposition of $G'$ of width $t-2$, contradicting the minimality of the counterexample $G$.  

    Thus $\beta_v$ has no vertex of $P$ except $w$. Hence the $f$-odd vertex $w''$ appears only in bags of components of $T-v$ not containing $r$, contradicting our choice of $w$.

\section{Homomorphism counting and oddomorphisms}
\label{sec:homcounting}

First we claim that to prove Theorem \ref{thm:maintheorem} it suffices to prove the following:

\begin{theorem}
\label{thm:reductiontheorem}
    For any $n \in \mathbb{N}$, there exists non-isomorphic graphs $H$ and $H'$ such that for every graph $F$, $\hom(F,H) \neq \hom(F,H')$ implies that $F$ contains a $K_{n}$-immersion.
\end{theorem}

The next proof is very similar to a theorem in \cite{oddomorphismpaper}. 
\begin{theorem}
    \label{thm:equivalence}
    Theorem \ref{thm:reductiontheorem} and Theorem \ref{thm:maintheorem} are equivalent.
\end{theorem}

\begin{proof}
    First suppose Theorem \ref{thm:reductiontheorem} is true and we are to prove Theorem \ref{thm:maintheorem}. Let $\mathcal{F}$ be a proper immersion-closed and union-closed class. As $\mathcal{F}$ is proper and union closed, there is some $n$ such that $K_{n}$ is not in $\mathcal{F}$. By Theorem \ref{thm:reductiontheorem}, there are non-isomorphic graphs $H$ and $H'$ such that $\hom(F,H) \neq \hom(F,H')$ implies that $F$ contains $K_{n}$ as an immersion. But then for all $K \in \mathcal{F}$, we have $\hom(K,H) = \hom(K,H')$, as $K_{n} \not \in \mathcal{F}$. Thus $H \equal_{\mathcal{F}} H'$, implying that $\equal_{\mathcal{F}}$ is not isomorphism.
    
    Now suppose that Theorem \ref{thm:maintheorem} is true. Let $n \in \mathbb{N}$ be given, and let $\mathcal{F}$ be the class of graphs with no $K_{n}$-immersion. By Theorem \ref{thm:maintheorem}, there exist non-isomorphic graphs $H,H'$ such that for all $K \in \mathcal{F}$, we have $\hom(K,H) = \hom(K,H')$. Hence, if $\hom(F,H) \neq \hom(F,H')$, then $F \not \in \mathcal{F}$, which implies that $F$ has a $K_{n}$-immersion by definition of $\mathcal{F}$. 
\end{proof}

Hence, our goal is to prove Theorem \ref{thm:reductiontheorem}.
We first extend the notion of oddomorphism from cliques to all graphs in order to utilize a theorem from \cite{oddomorphismpaper}.

\begin{definition}
    Let $f: V(G)\rightarrow V(H)$ be a homomorphism. A vertex $\vtxa \in V(G)$ is \underline{$f$-odd} (respectively, \underline{$f$-even}) if for every $\vtxb \in V(H) \cap N_{H}(f(\vtxa))$,  the number of vertices $\vtxc$ in $N_{G}(\vtxa)$ such that $f(\vtxc) = \vtxb$ is odd (respectively, even)
\end{definition}

\begin{definition} \label{def:oddomorphismGH}
Given graphs $G$ and $H$ and a homomorphism $f: V(G) \to V(H)$, we say that $f$ is an \underline{oddomorphism} from $G$ to $H$, denoted $G \odd H$, if:
\begin{enumerate}
    \item Every vertex of $G$ is either $f$-even or $f$-odd.
    \item For every vertex $v \in V(H)$, there is an odd number of vertices in $G$ whose image under $f$ is $v$ and are $f$-odd. 
\end{enumerate}
\end{definition}

\begin{definition}
   Let $F$ and $G$ be graphs and let $f$ be a homomorphism from $F$ to $G$. We say that $f$ is a \underline{weak oddomorphism} from $F$ to $G$ if there is a (not necessarily induced) subgraph $F'$ of $F$ such that the restriction of $f$ to $V(F')$ is an oddomorphism from $F'$ to $G$.
\end{definition}

\begin{theorem}[Theorem 3.13, \cite{oddomorphismpaper}]
\label{thm:robinsontheorem}
   Let $G$ be a connected graph. Then there are non-isomorphic graphs $G_{0}$ and $G_{1}$ such that $\hom(F, G_{0}) \geq \hom(F, G_{1})$ for every graph $F$ with strict inequality if and only if there exists a weak oddomorphism from $F$ to $G$. Moreover, if such a weak oddomorphism $f$ exists, then there is a connected subgraph $F'$ of $F$ such that $f|_{V(F')}$ is an oddomorphism from $F'$ to $G$. 
\end{theorem}

We note that in Theorem 3.13 in \cite{oddomorphismpaper}, the fact that $G_{0}$ and $G_{1}$ are non-isomorphic is not stated in the theorem. However this follows from Lemma 3.14 in \cite{oddomorphismpaper} and Theorem \ref{thm:lovasz}.

\begin{proof}[Proof of Theorem \ref{thm:reductiontheorem}]
Let $n \in \mathbb{N}$ be given. Let $G = K_{{n \choose 2}(7n+7)}$. By Theorem \ref{thm:robinsontheorem}, there are non-isomorphic graphs $G_{0},G_{1}$ such that $\hom(F,G_{0}) \geq \hom(F,G_{1})$ for every graph $F$ with strict inequality if and only if $F$ has a weak oddomorphism to $G$. As $G_{0}$ and $G_{1}$ are not isomorphic, by Theorem \ref{thm:lovasz} there is a graph $F$ such that $\hom(F,G_{0}) \neq \hom(F,G_{1})$. For all such graphs $F$, we have that $F$ admits a weak oddomorphism to $G$. Thus there is a subgraph $F'$ of $F$ such that $F' \odd G$.  But by Theorem \ref{thm:largeoddimplieslargeimmersion}, this implies $F'$ has a $K_{t}$-immersion, and hence $F$ has a $K_{t}$-immersion, thus implying Theorem \ref{thm:reductiontheorem}. 
\end{proof}

A similar proof using Theorem \ref{thm:treewidththm} shows that if $\mathcal{M}$ is the class of graphs with treewidth at most $k$, then $\equal_{\mathcal{M}}$ is not isomorphism. Since this approach was already taken in \cite{Neuen2024HomomorphismDistinguishing} and this is a known result, we omit this proof. 

\begin{ack}
    Benjamin Moore thanks Sabrina Lato and Leqi (Jimmy) Zhu for inspiring discussions on oddomorphisms. Youngho Yoo and Benjamin Moore thank the Pacific Institute of Mathematics for travel support to the University of Manitoba. Benjamin Moore thanks the Chilean research agency ANID for travel support. 
\end{ack}

\bibliographystyle{acm}
\bibliography{bib.bib}
\end{document}